\documentclass[12pt]{amsart}

\usepackage{amsmath,amssymb,lscape}
\usepackage{graphicx}
\usepackage{epstopdf}
\usepackage{enumerate}

\newtheorem{theorem}{Theorem}[section]

\newtheorem{lemma}[theorem]{Lemma}

\theoremstyle{definition}
\newtheorem{definition}[theorem]{Definition}

\theoremstyle{remark}

\numberwithin{equation}{section}

\def\bR{{\mathbb {R}}}

\def\pB{{\mathcal B}}

\newcommand\xqed[1]{%
  \leavevmode\unskip\penalty9999 \hbox{}\nobreak\hfill
  \quad\hbox{#1}}
\newcommand\tqed{\xqed{$\triangle$}}

\def\VB{{\mathcal {VB}}}
\def\VCD{{\mathcal {VCD}}}
\def\cVCD{{\mathcal {CVCD}}}

\begin{document}

\title{Distinguishing virtual braids in polynomial time}
\author{Oleg Chterental}
\email{oleg.chterental@mail.utoronto.ca}

\maketitle
\begin{center}\today\end{center}

\begin{abstract}
For $n \geq 2$ we describe an $O(l^3n)$-time algorithm that determines if a length $l$ virtual braid word in the standard presentation of the virtual braid group $\VB_n$ represents the trivial virtual braid.
\end{abstract}

\section{Introduction}

The virtual braid group $\VB_n$ with $n \geq 2$ strands has a presentation with generators $\sigma_1, \ldots, \sigma_{n-1}$ and $\tau_1, \ldots, \tau_{n-1}$ and relations $\sigma_i \sigma_j = \sigma_j \sigma_i$ for $|i-j|>1$, $\sigma_i \sigma_{i+1} \sigma_i=\sigma_{i+1} \sigma_i \sigma_{i+1}$ for $1 \leq i \leq n-2$, $\tau_i \tau_j = \tau_j \tau_i$ for $|i-j|>1$, $\tau_i \tau_{i+1} \tau_i=\tau_{i+1} \tau_i \tau_{i+1}$ for $1 \leq i \leq n-2$, $\tau_i^2=1$ for $1 \leq i \leq n-1$, $\sigma_i \tau_j = \tau_j \sigma_i$ for $|i-j|>1$ and $\tau_{i+1} \sigma_i \tau_{i+1} = \tau_i \sigma_{i+1} \tau_i$ for $1 \leq i \leq n-2$. It can be shown that the $\sigma_i$'s and the $\tau_i$'s generate embedded copies of the classical braid group $\pB_n$ and the symmetric group on $n$ symbols ${\rm Sym}_n$ respectively. A \textbf{virtual braid word} is a word over the alphabet of $\sigma_i$'s, $\sigma_i^{-1}$'s and $\tau_i$'s. We denote the length of a virtual braid word $\beta$ by $|\beta|$.

The word problem for the virtual braid groups $\VB_n$ was first shown to have a solution in Godelle and Paris \cite{GP} and a simpler approach was later found in Bellingeri, Cisneros de la Cruz and Paris \cite{BCP}. In \cite{C} we described the set $\VCD_n$ of virtual curve diagrams (vcd) and showed that there is a faithful action of $\VB_n$ on $\VCD_n$ which can be seen as an extension of the Artin action \cite{A} of $\pB_n$ on $F_n$, the rank $n$ free group. A virtual braid $\beta \in \VB_n$ can be recovered from the vcd $\beta \cdot I_n$ where $I_n$ is the trivial virtual curve diagram with $n$ curves and in particular a virtual braid word represents the trivial braid iff $\beta \cdot I_n=I_n$.

In the present paper we describe an encoding of virtual curve diagrams, in which each maximal collection of parallel arcs is replaced with a single multi-arc labelled with a positive integer. For classical braid words $\beta \in \pB_n$, it is known that the number of parallelity classes in a simplified curve diagram representing $\beta \cdot I_n$ is bounded above by a linear expression in $n$ independent of $|\beta|$, and each parallelity class has size that is bounded above by $M^{|\beta|}$ for some fixed positive integer $M$, independent of $|\beta|$ and $n$. The decimal representation of each parallelity class thus has at most $\frac{\log(M)}{\log(10)}|\beta|+1$ digits. This leads to an $O(|\beta|^2n)$ algorithm to calculate the encoding of a simplified representative of $\beta \cdot I_n$. This has been noted for example in the introduction of Dynnikov and Wiest \cite{DW}.

For a virtual braid word $\beta \in\VB_n$ we show that the number of parallelity classes in a simplified representative of $\beta \cdot I_n$ grows linearly in $|\beta|$ and $n$, and the size of each parallelity class is bounded above by $R^{|\beta|}$ for some fixed positive integer $R$ independent of $|\beta|$ and $n$. The decimal representation of the size of each parallelity class thus has at most $\frac{\log(R)}{\log(10)}|\beta|+1$ digits, and so in Theorem \ref{thmmain} we see that an encoding of a simplified representative of $\beta \cdot I_n$ can be calculated in $O(|\beta|^3n)$ time. 

\section{Condensed virtual curve diagrams}

We will assume the following terminology from \cite{C}: virtual curve diagrams, points, curves, $<_P$, $<_C$, upper points, base points, adjacent points, arcs, over arcs, under arcs, base arcs, terminal arcs, arcs enclosing a point, crossing arcs, free under arcs, $T$-moves, $B$-moves, and simplified vcds. If $a,b$ are points that bound an arc in a vcd and $a <_C b$ recall that we use the notation $(a,b)$ to denote the arc. 

\begin{definition}[Condensed virtual curve diagrams] We will say two over or two under arcs $(a,b)$ and $(c,d)$ are \textbf{parallel} if $a$ and $c$ (resp. $d$) are adjacent, $b$ and $d$ (resp. $c$) are adjacent, and the arcs do not cross. The reflexive transitive closure of the above definition will also be called parallelity. Each base arc we consider to be parallel only to itself. We will call a parallelity class $A$ a \textbf{condensed arc}. The number of arcs in a condensed arc $A$ will be refered to as its \textbf{weight} and will be denoted $|A|$.

If $(a,b)$ is an arc we say the points $a$ and $b$ are \textbf{incident} to the arc. A point is said to be incident to a condensed arc $A$ if it is incident to an arc in $A$. Two condensed arcs are incident if they are incident to the same point. It is necessarily the case that a pair of incident condensed arcs includes a condensed over arc and either a condensed under or condensed base arc. Let $a_1<_P \ldots <_P a_r$ and $b_r <_P \ldots <_P b_1$ be the points incident to a condensed over or under arc $A$ of weight $r$, where each pair $\{a_i,b_i\}$ determines an arc in $A$ (either $(a_i,b_i)$ or $(b_i,a_i)$ depending on if $a_i <_C b_i$ or $b_i <_C a_i$). The points $a_1$, $a_r$, $b_r$ and $b_1$ will be refered to as the \textbf{corners} of the condensed arc $A$. A condensed arc with weight one has only two corners. We denote by $A\urcorner$ the set of corners of $A$. For a condensed arc $A$ let $type(A) \in \{o,u,b\}$ indicate whether $A$ is a condensed over, under or base arc respectively.

\begin{figure}
\begin{center}
\includegraphics[scale=.85]{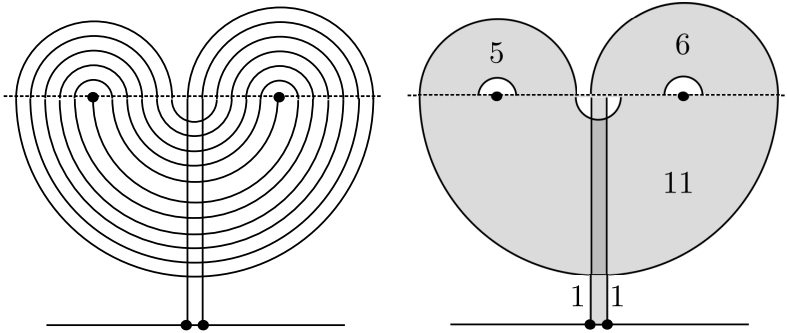}
\caption{A condensed vcd with $5$ condensed arcs and $12$ corners.}\label{figcondvcd}
\end{center}
\end{figure}

It should be evident that one may completely specify a virtual curve diagram by deleting all points that are not terminal points or corners, all arcs that are not incident to corners, and labelling condensed arcs by their weight, as in Figure \ref{figcondvcd}. A bit more formally, let $E=(P,<_P,<_C)$ be a vcd with $n$ curves (recall $P$ is the set of points, $(P,<_P)$ is a strict partial order consisting of two disjoint chains $U$ and $B$ containing the upper points and the base points, and $(P,<_C)$ is a strict partial order containing $n$ disjoint chains representing the $n$ curves of $E$, see \cite[Definition 2.1]{C} for more details) with $m$ condensed arcs $A_1,\ldots,A_m$. The \textbf{condensed virtual curve diagram} $D$ with $n$ curves, whose \textbf{underlying} vcd is $E$, is given by the tuple $D=(T, \{A_i\urcorner\}_{i=1}^m,\{|A_i|\}_{i=1}^m,<_{P\urcorner}, \{type(A_i)\}_{i=1}^m)$ where $T \subset P$ is the set of terminal points of $E$, and $<_{P\urcorner}$ is the restriction of $<_P$ to $P\urcorner=T \cup \bigcup_i (A_i\urcorner)$. We denote by $\cVCD_n$ the set of condensed virtual curve diagrams with $n$ curves.
\tqed
\end{definition}

\section{Condensed $T$ and $B$-moves}

Let $(a,b)$ be an over or under arc with $a$ and $b$ adjacent. Recall the $T$ and $B$-moves for vcds in Figure \ref{figtb}. If $b$ is a terminal point then a $T$-move may be performed by deleting $b$. Otherwise a $B$-move may be performed by deleting both points $a$ and $b$. A $T$ or $B$-move that reduces the number of points is called a simplification and a vcd admitting no simplifying $T$ or $B$-moves is simplified. Any vcd is equivalent via $T$ and $B$-moves to a unique simplified vcd. This is analogous to the fact that any element of a free group has a unique freely reduced representative. A condensed vcd will be called \textbf{simplified} if its underlying vcd is simplified.

\begin{figure}
\begin{center}
\includegraphics[scale=.85]{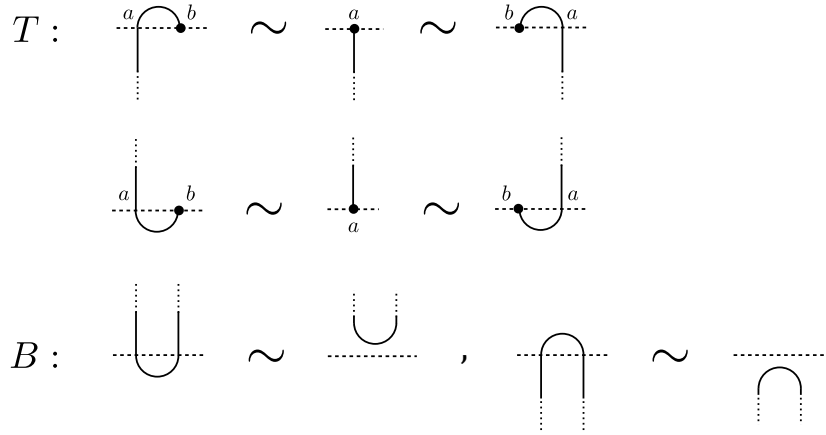}
\caption{The $T$ and $B$-moves for vcds.}\label{figtb}
\end{center}
\end{figure}

\begin{figure}
\begin{center}
\includegraphics[scale=.85]{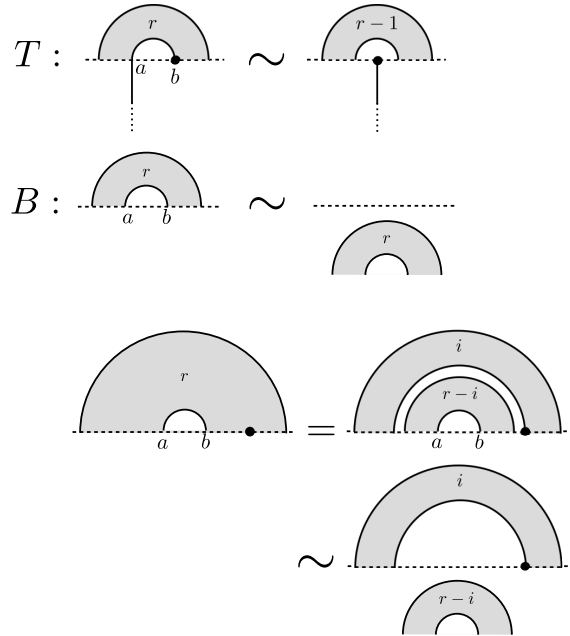}
\caption{The condensed $T$ and $B$-moves for condensed vcds (horizontal and vertical reflections not shown).}\label{figcondtb}
\end{center}
\end{figure}

\begin{definition}[Condensed $T$ and $B$-moves]\label{defcondtb}
Condensed versions of the $T$ and $B$-moves are described in Figure \ref{figcondtb}. Let $a_1<_P \ldots <_P a_r$ and $b_r <_P \ldots <_P b_1$ be the points incident to a condensed over or under arc $A$ of weight $r$. If $a_r$ and $b_r$ are adjacent and one of $a_r$ or $b_r$ is a terminal point then a simplifying condensed $T$-move may be performed.

If $a_r$ and $b_r$ are not terminal points then a simplifying condensed $B$-move may be performed, of which there are two versions. If the condensed arc $A$ contains no terminal points at all then a simplifying condensed $B$-move may be performed by deleting all points of $A$. Otherwise choose the largest $1 \leq i < r$ such that one of $a_i$ or $b_i$ is a terminal point. A simplifying condensed $B$-move may be performed by deleting all points $a_j$ and $b_j$ for $i<j \leq r$.
\tqed
\end{definition}

Recall that an under arc in a vcd is free if it does not cross any other under or base arc and does not enclose any pair of crossing under arcs. A base arc will be said to be free if it does not cross any under or base arc. Similarly we define a condensed under or base arc to be \textbf{free} if it consists of free under or base arcs. The following lemma summarizes some elementary facts about the condensed $T$ and $B$-moves and condensed vcds in general.

\begin{lemma}\label{lemcond}
\begin{enumerate}[a)]
\item\label{cond1} A condensed vcd is simplified iff it admits no simplifying condensed $T$ or $B$-moves.
\item\label{cond2} If $D_1$ and $D_2$ are simplified condensed vcds related by condensed $T$ and $B$-moves then $D_1=D_2$.
\item\label{cond3} If $D \in \cVCD_n$ is a simplified condensed vcd then $D$ has at most $2n-1$ condensed over arcs, at most $2n-1$ free condensed under arcs and at most $n$ free condensed base arcs.
\end{enumerate}
\end{lemma}
\begin{proof}
\begin{enumerate}[a)]
\item Recall that a simplified condensed vcd is by definition a condensed vcd whose underlying vcd is simplified. We must show this is equivalent to the condensed vcd admitting no simplifying condensed $T$ or $B$-moves.

First we show that a simplified condensed vcd $D$ admits no simplifying condensed $T$ or $B$-moves. Let $E$ be the underlying simplified vcd of $D$. Assume for a contradication that $D$ admits a simplifying condensed $T$ or $B$-move. By Definition \ref{defcondtb}, there must be an arc $(a_r,b_r)$ in $E$ for which $a_r$ and $b_r$ are adjacent. It follows that $E$ admits a simplifying $T$ or $B$-move contradicting the fact that it is a simplified vcd. Thus $D$ cannot admit a simplifying condensed $T$ or $B$-move.

Assume now that the condensed vcd $D$ admits no simplifying condensed $T$ or $B$-move and let $E$ be its underlying vcd. We must show that $D$ is simplified, which by definition means we must show that $E$ is simplified. Assume that $E$ is not simplified. Thus there must be an arc $(a,b)$ in $E$ for which $a$ and $b$ are adjacent. By Definition \ref{defcondtb}, $D$ admits a simplifying condensed $T$ or $B$-move along the condensed arc containing $(a,b)$, contradicting the assumption that $D$ admits no simplifying condensed $T$ or $B$-moves. Thus $E$ is simplified and so $D$ is simplified as well.

\item Assume $D_1$ and $D_2$ are simplified condensed vcds related by condensed $T$ and $B$-moves. Their underlying vcds $E_1$ and $E_2$ are both simplified and related by $T$ and $B$-moves. Thus $E_1=E_2$ and so $D_1=D_2$.

\item We use the fact that the maximum size of a collection of non-parallel disjoint simple arcs properly embedded in $\{(x,y) \in \bR^2: y \geq 0\}\setminus \{(i,0)\}_{i=1}^k$, the upper half-plane minus $k$ points along the $x$-axis, is $2k-1$. This is straightforward to prove by induction on $k \geq 1$.

Note that by the definition of a vcd, no two over arcs may cross. Let $D \in \cVCD_n$ be a simplified condensed vcd and $E \in \VCD_n$ its underlying simplified vcd. The vcd $E$ has $n$ terminal points and so has $k \leq n$ terminal points incident to terminal under arcs. Since $E$ is simplified, any over arc must enclose at least one terminal point, and specifically at least one terminal point incident to a terminal under arc (other wise consider an innermost over arc that does not enclose a terminal point incident to a terminal under arc, and simplify). Two over arcs of $E$ are parallel exactly when they enclose the same terminal points that are incident to terminal under arcs. Using the above fact about non-parallel collections of disjoint simple arcs in a half-plane minus $k$ points along the $x$-axis, it follows that a simplified condensed vcd can have at most $2k-1 \leq 2n-1$ condensed over arcs. This reasoning applies to free condensed under arcs for the same reasons. All condensed vcds in $\cVCD_n$ contain at most $n$ free condensed base arcs since they have $n$ condensed base arcs.
\end{enumerate}
\end{proof}

\section{The action of $\VB_n$ on $\cVCD_n$}

Let $D \in \cVCD_n$ be a simplified condensed virtual curve diagram and let $E \in \VCD_n$ be its underlying vcd. Using \cite[Section 3]{C} we may directly calculate $g \cdot E$ for any length one virtual braid word $g$. We define $g \cdot D$ to be the unique condensed virtual curve diagram with underlying vcd $g \cdot E$. If $\beta$ is a virtual braid word we define $\beta \cdot D$ using the rule $(\alpha g)\cdot D=\alpha \cdot (g \cdot D)$, as for vcds.

To analyze the possible simplifying condensed $T$ and $B$-moves that arise when calculating $g \cdot D$ we introduce some terminology. Let $t_1 <_P \ldots <_P t_n$ be the terminal points of $E$. We say an upper point $a$ of $E$ is \textbf{right-veering} if it is incident to an over arc with some point $a <_P b$ and left-veering if it is incident to an over arc with some point $b <_P a$ (assuming it is incident to an over arc). We say a simplifying $B$-move in some vcd along an under arc $(a,b)$ with $a$ and $b$ adjacent is \textbf{superficial} if $a$ and $b$ are incident to over arcs and veer in different directions. We say a simplifying condensed $B$-move in some condensed vcd is \textbf{superficial} if each of the simplifying $B$-moves in the sequence is superficial.

For a simplified vcd $E$, all left-veering points precede all right-veering points in each interval $[t_i,t_{i+1}]_{<_P}$ along the $<_P$ order. If $(t_i,t_{i+1})_{<_P}$ contains a left-veering point we let $u_i$ be the $<_P$-largest left-veering point otherwise let $u_i=t_i$. Let $v_i$ be the $<_P$-smallest point in $(u_i,t_{i+1}]_{<_P}$. We refer to the points $u_i$ and $v_i$ as the \textbf{center} points of $[t_i,t_{i+1}]_{<_P}$.

\begin{lemma}\label{lemsimp}
\begin{enumerate}[a)]
\item\label{simp1} Let $D \in \cVCD_n$ be simplified and let $g=\sigma_i^\pm$ be a braid generator or its inverse. The condensed vcd $g \cdot D$ can be simplified by applying at most one simplifying condensed $B$-move and at most one simplifying condensed $T$-move. If $D$ contains $m$ condensed arcs then the simplified representative of $g \cdot D$ has at most $m+4n$ condensed arcs.
\item\label{simp2} If $D \in \cVCD_n$ is simplified then $\tau_i \cdot D$ can be simplified in at most three simplifying condensed $B$-moves and four simplifying condensed $T$-moves. If $D$ contains $m$ condensed arcs then the simplified representative of $\tau_i \cdot D$ contains at most $m+17n$ condensed arcs.
\item\label{simp3} If $\beta \in \VB_n$ is a virtual braid word then the simplified representative of the condensed vcd $\beta \cdot I_n$ has at most $(1+17|\beta|)n$ condensed arcs, each with weight at most $6^{|\beta|}$.
\end{enumerate}
\end{lemma}
\begin{proof}
\begin{enumerate}[a)]
\item Assume without loss of generality that $g=\sigma_i$. Due to the definition of the action of $\VB_n$ on $\VCD_n$, even the trivial element $1_{\VB_n}$ acting on $E$ produces a vcd $1_{\VB_n} \cdot E$ that is not simplified, but requires some superficial simplifying $B$-moves and at most $n$ simplifying $T$-moves. We will use a slightly modified action of braid generators on simplified vcds that is a bit more striaghtforward. The vcd associated to $g \cdot E$ will be the obvious vcd obtained by having the terminal point $t_i$ travel over and to the right of $t_{i+1}$ along a semi-circular path, dragging along any over arcs in the way, and coming to rest between the center points of $[t_{i+1},t_{i+2}]_{<_P}$ (or if $i=n-1$, coming to rest after the $<_P$-largest upper point). Note that the action of a braid generator affects only over arcs, and free under arcs in $E$. The union of all over arcs and free under arcs forms the curve diagram for a classical braid in $\pB_n$ with some additional curves. See \cite[Figure 20 and Proposition 3.2]{C} for an example of such a factorization. Thus non-free condensed arcs in $E$ are the same as those in $g \cdot E$ and in any vcd related to $g \cdot E$ by $T$ and $B$-moves.

The vcd associated to $g \cdot E$ described in the previous paragraph might not be simplified, and may require a sequence of simplifying $B$-moves followed by a single simplifying $T$-move. The simplifying $B$-moves involve free under arcs enclosing only $t_i$ in $E$ and the simplifying $T$-move, which may or may not be necessary, will involve a terminal under arc enclosing only $t_i$ and incident to $t_{i+1}$ in $E$ (see \cite[Figure 24]{C}). In the condensed vcd $g \cdot D$ this corresponds to a single simplifying condensed $B$-move and a simplifying condensed $T$-move. The simplifying condensed $B$-move may be superficial, or not.

\begin{figure}
\begin{center}
\includegraphics[scale=.9]{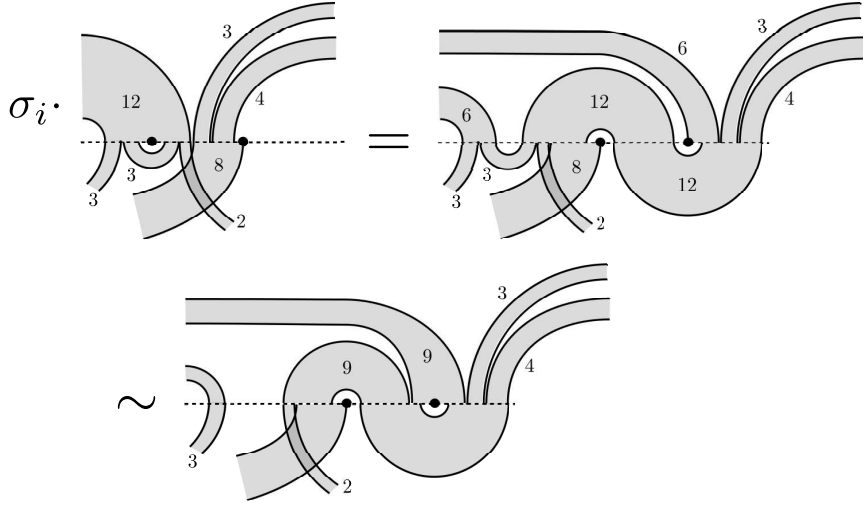}
\caption{The splitting of a condensed over arc.}\label{figbraidaction}
\end{center}
\end{figure}

When calculating $g \cdot D$, if $t_i$ is incident to a condensed over arc $A$, it is possible that $A$ will split into two condensed over arcs in $g \cdot D$, Figure \ref{figbraidaction} describes this situation (note that Lemma \ref{lemcond}\ref{cond3} still holds).

The simplified condensed representative of $g \cdot D$ will have the same number of non-free condensed under arcs, non-free condensed base arcs and free condensed base arcs as $D$. The only changes occur in the number of condensed over and free condensed under arcs. By Lemma \ref{lemcond}\ref{cond3} there are at most $2n-1$ of each type, thus the simplified representative of $g \cdot D$ will have at most $m+(2n-1)+(2n-1)<m+4n$ condensed arcs.

\item First we recall in Figure \ref{figpermarcs} the action of $\tau_i$ on over arcs in various positions in $E$, before any simplifications.
\begin{figure}
\begin{center}
\includegraphics[scale=.8]{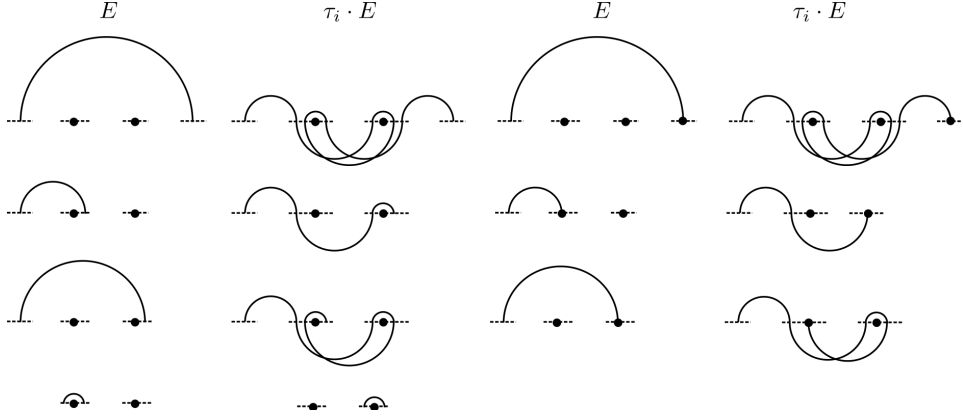}
\caption{The action of $\tau_i$ on over arcs in $E$, before simplifications.}\label{figpermarcs}
\end{center}
\end{figure}
Two parallel over arcs in $E$ that are not incident to $t_i$ or $t_{i+1}$ will remain parallel in $\tau_i \cdot E$, thus Figure \ref{figpermarcs} applies to condensed over arcs that are not incident to $t_i$ or $t_{i+1}$ as well. If a condensed over arc in $D$ is incident to $t_i$ or $t_{i+1}$ or both it will split, as in Figure \ref{figpermsplit}. In any case we see that a condensed over arc may break into a group of at most seven condensed arcs.

\begin{figure}
\begin{center}
\includegraphics[scale=.9]{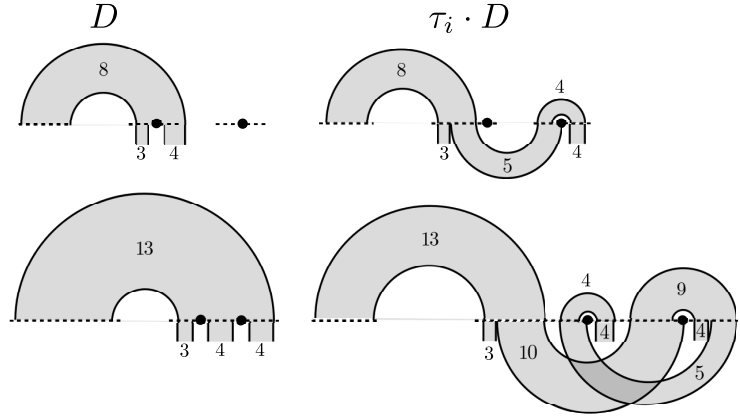}
\caption{The action of $\tau_i$ on condensed over arcs in $D$ incident to $t_i$ or $t_{i+1}$, before simplifications.}\label{figpermsplit}
\end{center}
\end{figure}

The action of $\tau_i$ on under arcs in $E$ only moves their endpoints and does not break them up into sequences of arcs as in Figure \ref{figpermarcs}. An under arc in $E$ is moved by $\tau_i$ only if it is incident to $t_i$ or $t_{i+1}$, or if it is incident to a point in $(t_{i-1},t_i)_{<_P}$ that is right-veering, or any point in $(t_i,t_{i+1})_{<_P}$, or a point in $(t_{i+1},t_{i+2})_{<_P}$ that is left-veering. If a condensed under arc is incident to both center points of one of the intervals $[t_{i-1},t_{i}]_{<_P}$, $[t_i,t_{i+1}]_{<_P}$ or $[t_{i+1},t_{i+2}]_{<_P}$, then it will split into two condensed under arcs. There can be at most three such condensed under arcs in $D$. No other condensed under arc is split by the action of $\tau_i$.

The number of condensed base arcs always remains fixed at $n$. Based on the above, the naive condensed vcd representing $\tau_i \cdot D$ will have at most $m+7(2n-1)+3<m+15n$ condensed arcs.

The naive representative of $\tau_i \cdot E$ may require superficial simplifying $B$-moves along parallel arcs nested at the centers of the intervals $(t_{i-1},t_{i})_{<_P}$, $(t_i,t_{i+1})_{<_P}$, and $(t_{i+1},t_{i+2})_{<_P}$ followed by at most four possible parallel simplifying $T$-moves along the terminal under arcs incident to $t_{i-1}$, $t_i$, $t_{i+1}$ and $t_{i+2}$. The condensed vcd $\tau_i \cdot D$ consequently may thus require three superficial simplifying condensed $B$-moves, one in the center of each aforementioned intervals, followed by at most four simplifying condensed $T$-moves along terminal under arcs also in the centers incident to the terminal points $t_{i-1}$, $t_i$, $t_{i+1}$ and $t_{i+2}$. It is conceivable that this simplification process increases the number of condensed over arcs, thus a simplified representative for $\tau_i \cdot D$ will have at most $m+15n+(2n-1)<m+17n$ condensed arcs.

\item Since $I_n$ has $n$ condensed arcs (the base arcs), applying Lemmas \ref{lemsimp}\ref{simp1} and \ref{lemsimp}\ref{simp2} we see that the simplified repersentative of $\beta \cdot I_n$ has at most $n+17n|\beta|=(1+17|\beta|)n$ condensed arcs. 

If a simplified vcd $E$ has $M$ non-terminal upper points then it has at most $M$ over arcs. In calculating the naive vcd representing $\sigma_i \cdot E$, each over arc may undergo a single non-simplifying $B$-move as it is dragged through the upper line contributing two extra points. If the terminal arc incident to $t_i$ is an under or base arc, it may contribute a single point as well. Thus the naive representative of $\sigma_i \cdot E$ may have at most $2M+1$ non-terminal upper points, and the simplified representative will have at most $2M+1$ non-terminal upper points since simplifications decrease the number of non-terminal upper points. Similar reasoning applies to $\sigma_i^{-1} \cdot E$.

In calculating a naive representative for $\tau_i \cdot E$, under arcs are simply moved around. Each over arc may split into a sequence of at most seven arcs as in Figure \ref{figpermarcs}, contributing an extra six points. Thus if $E$ has $M$ non-terminal upper points, then it has at most $M$ over arcs and so $\tau_i \cdot E$ has at most $6M$ non-terminal upper points.

Since $I_n$ has $0$ non-terminal upper points, a simplified representative for $\beta \cdot I_n$ then has at most $(f_2)^{|\beta|-1}(f_1(0))=6^{|\beta|-1}$ non-terminal upper points (where $f_1:M \mapsto 2M+1$ and $f_2:M \mapsto 6M$). The weight of each condensed base arc is $1$. Any over or under arc must be incident to at least one non-terminal upper point, thus each arc in a condensed over or condensed under arc must be incident to at least one non-terminal upper point. Thus the weight of each condensed over or condensed under arc must be at most $6^{|\beta|-1}<6^{|\beta|}$.
\end{enumerate}
\end{proof}

\begin{theorem}\label{thmmain}
If $\beta \in \VB_n$ is a virtual braid word then the simplified representative of the condensed virtual curve diagram $\beta \cdot I_n$ can be calculated in $O(|\beta|^3n)$ time.
\end{theorem}
\begin{proof}
Note that a simplified condensed vcd $D$ with $n$ curves and $m$ condensed arcs $A_1,\ldots,A_m$ can be encoded in $O(m\log(r))$ space where $r=\max_i \{|A_i|\}$. A simplifying condensed $T$ or $B$-move can be performed by scanning through the entire structure once, i.e. $O(m \log(r))$ time.

By Lemma \ref{lemsimp}\ref{simp1} and \ref{lemsimp}\ref{simp2} the number of simplifications in calculating a simplified condensed vcd $g \cdot D$ from a simplified condensed vcd $D$ is constant, and so takes $O(m \log(r))$ time.

For a virtual braid word $\beta \in \VB_n$ the simplified condensed representative $D$ of $\beta \cdot I_n$ has by Lemma \ref{lemsimp}\ref{simp3} at most $(1+17|\beta|)n$ condensed arcs each with weight at most $6^{|\beta|}$. Thus $D$ can be encoded in $O(|\beta|^2n)$ space. Given a length one virtual braid word $g \in \VB_n$, calculating a simplified condensed representative for $(g \beta) \cdot I_n=g \cdot D$ takes $O(|\beta|^2n)$ time. Thus a simplified representative of the condensed vcd $\beta \cdot I_n$ can be calculated in $O(|\beta|^3n)$ time.
\end{proof}


\begin{thebibliography}{99}
\bibitem[Ar]{A}
E. Artin, \textit{Theory of Braids}, Annals of Mathematics, Second Series, Vol. 48, No. 1 (Jan., 1947), 101--126.
	
\bibitem[BCP]{BCP}
P. Bellingeri, B.A. Cisneros de la Cruz, L. Paris, \textit{A simple solution to the word problem for virtual braid groups}, Pacific Journal of Mathematics 283 (2016), 271--287.

\bibitem[C]{C}
O. Chterental, \textit{Virtual braids and virtual curve diagrams}, J. Knot Theory Ramifications, Volume 24, Issue 13, November 2015.

\bibitem[DW]{DW}
I. Dynnikov, B. Wiest, \textit{On the complexity of braids}, J. Eur. Math. Soc. 9 (2007), 801--840.

\bibitem[GP]{GP}
E. Godelle, L. Paris, \textit{K($\pi$, 1) and word problems for infinite type Artin–Tits groups, and applications to virtual braid groups}, Mathematische Zeitschrift, December 2012, Volume 272, Issue 3-4, 1339--1364
\end{thebibliography}
\end{document}